\documentclass[11pt]{amsart}
\usepackage[pagewise]{lineno}
\usepackage{adjustbox}

\usepackage{nicefrac}
\usepackage[latin1]{inputenc}
\usepackage{array,booktabs}
\usepackage{tikz}
\usetikzlibrary{positioning}
\usetikzlibrary{arrows}
\usepackage{pifont}
\usepackage{amssymb,amsmath,amsthm,amsfonts}
\usepackage{comment,enumerate}
   \usepackage{tikz-cd}
     \usetikzlibrary{cd}
\usepackage{todonotes}
\usepackage{xcolor}

\usepackage[colorlinks=true, citecolor=blue, linkcolor=magenta, menucolor=magenta, urlcolor=magenta]{hyperref}

\newcommand\ant{\nicefrac12}
\newcommand{\alg}{\mathbf}
\newcommand{\class}{\mathsf}
\newcommand{\logic}{\textsc}

\newcommand{\pair}[2]{\langle #1, #2 \rangle}

\makeatletter
\providecommand*{\Dashv}{%
  \mathrel{%
    \mathpalette\@Dashv\vDash
  }%
}
\newcommand*{\@Dashv}[2]{%
  \reflectbox{$\m@th#1#2$}%
}
\makeatother\newcommand{\iso}{\cong}
\newcommand{\equals}{\approx}

\newcommand{\assign}{\mathrel{:=}}
\newcommand{\boldrho}{\boldsymbol{\rho}}
\newcommand{\boldtau}{\boldsymbol{\tau}}

\DeclareMathOperator{\Alg}{Alg}

\newcommand{\Tarski}{\widetilde{\Omega}}

\DeclareMathOperator{\Ker}{Ker}

\newcommand{\AlgH}{\mathbb{H}}

\newcommand{\AlgI}{\mathbb{I}}
\newcommand{\AlgS}{\mathbb{S}}
\newcommand{\AlgP}{\mathbb{P}}
\newcommand{\AlgPS}{\mathbb{P}_{\mathrm{SD}}}

\newcommand{\AlgPU}{\mathbb{P}_{\mathrm{U}}}

\newcommand{\ddtset}{I}

\newlength\auxskip

\theoremstyle{theorem}
\newtheorem{theorem}{Theorem}[section]
\newtheorem{lemma}[theorem]{Lemma}
\newtheorem{proposition}[theorem]{Proposition}
\newtheorem{corollary}[theorem]{Corollary}
\newtheorem{fact}[theorem]{Fact}

\theoremstyle{definition}
\newtheorem{definition}[theorem]{Definition}
\newtheorem{example}[theorem]{Example}
\newtheorem{assumption}[]{Assumption}

\title{The external version of a subclassical logic}

\author{Massimiliano Carrara}
\address{University of Padua, FISPPA Department}
\email{massimiliano.carrara@unipd.it}

\author{Michele Pra Baldi}
\address{University of Padua, FISPPA Department}
\email{michele.prabaldi@unipd.it}
\begin{document}

\maketitle

\begin{abstract}
A three-valued logic  $\logic{L}$ is subclassical when it is defined by a single matrix having the classical two-element matrix as a subreduct. In this case, the language of $\logic{L}$ can be expanded with special unary connectives, called \emph{external operators}. The resulting logic $\logic{L}^{e}$ is  the \emph{external version} of $\logic{L}$, a notion originally introduced by D. Bochvar in 1938 with respect to his weak Kleene logic. In this paper we study the semantic properties of the external version of a three-valued subclassical logic $\logic{L}$. We determine sufficient and necessary conditions to turn a model of $\logic{L}$ into a model of $\logic{L}^{e}$. Moreover, we establish some distinctive semantic properties of $\logic{L}^{e}$.

 \end{abstract}


\section{Introduction}
In the philosophy of logic debate departures from classical reasoning has been motivated by specific phenomena such as paradoxes, partial information and
vagueness. 
But  ``paradoxical sentences - as G. Priest argues - seem to be a fairly small proportion of the sentences we reason with. It would seem plausible to claim that in our day-to-day reasoning we (quite correctly) presuppose that we are not dealing with paradoxical claims'' (\cite[p. 235]{Priestfirst}). And when we are not dealing with these phenomena, we want to be able to reason classically; indeed classical logic seems to be perfectly in order as it is. Now, the problem is: how can we secure we reason ``classically'' when we are not dealing with these non-normal phenomena? Or, put differently: how can we recover specific sets of classical inferences within a non-classical propositional logic, such as in many-valued logics? A widely recognized approach is to enrich the language with new unary connectives, commonly referred to as normality, classicality, consistency, or determinedness operators.

This is precisely the direction originally pursued by N. da Costa, the founder of the well-established tradition of the Logics of Formal Inconsistency within the field of paraconsistent logics (we refer the reader to \cite{carnielli2007logics} for a comprehensive survey of this tradition). Subsequently, similar intuitions  led J. Marcos to the development of the Logics of Formal Undeterminedness \cite{marcos2005nearly}. More recently, D. Szmuc advanced this line of research by investigating  consistency and determinedness operators for a specific family of logics of variable inclusion (also known as infectious logics) \cite{SzmucMS}. 
However, the theoretical origins of this approach can be traced back to 1938, when the Russian logician D. Bochvar  expanded his three-valued logic with the so-called 
\emph{external} connectives \cite{Bochvar}.  Bochvar's three-valued logic $\logic{B}$ is formulated within the classical language $\langle\land,\lor,\neg\rangle$ (see Example \ref{example: subclassical logics}), and it was historically studied in connection with paradoxes.
Despite its modest success in dealing with paradoxes, this logic has been experiencing a resurgence in recent years. On the formal side, it has been explored within the framework of logics of variable inclusion \cite{Bonziobook}; in philosophy it has been  applied to the debate on subject-matter \cite{beall2016off, ferguson2017meaning}; and in epistemology it has been employed in belief revision theory \cite{CarraraEpistemic,carrara2024paraconsistent}. 
In this logic, assigning the non-classical truth value $\ant$ to a sentence $\varphi$ traditionally means that $\varphi$ is considered meaningless (or off-topic), which is equivalent to the evaluation of an atom appearing in $\varphi$ as $\ant$.
 Consequently, $\logic{B}$ is a theoremless logic, in the sense that every classical theorem (think of $\varphi\lor\neg\varphi$) admits a 
 counterexample by simply evaluating one of the variables appearing in $\varphi$ to $\ant$. In less formal terms, when some of the information contained in $\varphi \lor \neg \varphi$ is meaningless, the entire sentence also becomes meaningless. This distinctive feature is known as the \emph{contamination principle}.
Nonetheless, we may want our logical language to express the idea that ``$\varphi\lor\neg\varphi$ is true whenever it is meaningful''.
This is precisely what Bochvar achieved by enriching the language with new unary  connectives $\Delta_{0}$,$\Delta_{\ant}$,$\Delta_{1}$, called \emph{external operators}. Their intended interpretation is ``being meaningful and false'', ``being meaningless'', ``being meaningful and true'' respectively. The same operators were independently introduced, under different names, by  D'Ottaviano and da Costa in 1970 \cite{d1970probleme} to  formalize  the paraconsistent logic $\logic{J}_{3}$ (See Example \ref{example: subclassical logics}), further refined in collaboration with Epstein few years later \cite{Epstein1990}. 

These external operators allow us to express semantic features that the \emph{internal} language $\langle\land,\lor,\neg\rangle$ of $\logic{B}$ cannot capture.  Semantically, they  provide a yes/no answer, taking values among $\{0,1\}$, to whether their arguments are classical/non-classical or true/false.
 Expanding $\logic{B}$ with these external operators yields Bochvar external logic ($\logic{B}^{e}$). In this system, the statement ``$\varphi\lor\neg\varphi$ is true whenever it is meaningful" corresponds to the fact that $\Delta_{1}\varphi\lor\neg\Delta_{1}\varphi$ is a theorem of $\logic{B}^{e}$. 

 A closer look into investigations of Bochvar through the lens of modern algebraic logic shows that, when minimal conditions are met, the concept of \emph{external version} of a logic does not rely on specific properties of $\logic{B}$, but it can be  applied to a wide family of three-valued logics, which we call \emph{subclassical} (see Definition \ref{def: sublcassical logic}). 
 
Despite extensive logico-philosophical work on consistency, determinedness, and modal-like operators, a unified approach to external operators is still lacking. Likewise, the external version of subclassical logic has not been systematically studied using modern algebraic logic. This paper lays the groundwork for such a study.

A further motivation for this investigation is that some well-known logics, besides $\logic{B}^{e}$, can be understood as the external version of a subclassical logic. One example is three-valued \L ukasiewicz logic $\logic{\L}_{3}$, that can be equivalently 
formulated by replacing the implication $\to_{\logic{\L}}$ with the operator $\Delta_{1}$ in the three-element \L ukasiewicz chain (see Figure \ref{figure: algebras subclassical}).
In this algebra, $\Delta_{1}$ is known as the Monteiro-Baaz operator, traditionally understood as a modal operator of necessity (see \cite{cignoli1965boolean,cignoli2007algebras} and \cite{cintula2011handbook}). 
Thus, $\logic{\L}_{3}$ can be viewed as the external version of Strong Kleene logic $\logic{K}_{3}$, in symbols $\logic{K}_{3}^{e}=\logic{\L}_{3}$. 
Likewise, the paraconsistent logic $\logic{J}_{3}$ corresponds to the external version of the Logic of Paradox $\logic{LP}$.  Despite this, the connection between these logics has not yet been explored within the conceptual framework considered here.

The key takeaway is that several well-known logics, which might initially seem unrelated, actually converge as external versions of certain three-valued logics.
This suggests that the notion of external version of a logic might provide a unifying framework for studying a broad class of logics.

In this paper, we tackle two central questions related to the external version  $\logic{L}^{e}$ of a logic $\logic{L}$. First, in Section \ref{sec: semantic meaning}, we  determine necessary and sufficient semantic conditions to define an external operator over a model of $\logic{L}$ (Theorem \ref{thm: representation via conunclei}). For the special case of Bochvar logic and Strong Kleene logic, this question has been addressed in \cite{bonzio2023structure} and \cite{cignoli1965boolean}. 
 We prove that not every model of a subclassical logic supports the definability of an external operator, and we establish a useful criterion to detect when this is the case (Corollary \ref{cor: on separting}). 
Furthermore, in Section \ref{sec: properties external} we provide an overview of the logical properties that are recaptured in the step from a logic to its external version. This is the case of well-known and desirable properties of classical logic, such as algebraizability and the deduction theorem.

\section{Preliminaries} \label{sec: preliminaries}

For unexplained terminology the reader can refer to \cite{Font16} or \cite{Cz01}. In this paper we will mainly deal with (expansions of) two propositional languages: the classical language $\langle\land,\lor,\neg\rangle$ of type $\langle 2,2,1\rangle$ and  the language $\langle\land,\lor,\neg,\Delta_{1}\rangle$ of type $\langle 2,2,1,1\rangle$. We do not consider constant symbols in the language. Once a propositional language is fixed, by $Fm$ we denote the set of formulas built over it, and by $\alg{Fm}$ the absolutely free algebra with universe $Fm$. When necessary, we will write $\alg{Fm}_{\mathcal{L}}$ to empathize that we are referring to the language $\mathcal{L}$.  We will often omit the notational distinction between an algebra $\alg{A}$ and its underlying universe $A$, thus writing $a\in\alg{A}$ in place of $a\in A$. Whenever $\alg{A},\alg{B}$ are algebras of the same type (or \emph{similar})  and $h:\alg{A}\to\alg{B}$ is a homomorphism, we set $h[X]=\{b\in\alg{B}: b=h(a) \text{ for some } a\in X\}$, for $X\subseteq\alg{A}$. The dual notion applies to $h^{-1}[Y]$, where $Y\subseteq\alg{B}$. For a class of algebras $\class{K}$ we denote by $\AlgI,\AlgS,\AlgH,\AlgP,\AlgPS,\AlgPU (\class{K})$ the closure of $\class{K}$ under isomorphic copies, subalgebras, homomorphic images, products and subdirect products and ultraproducts.

 Given a propositional language, a logic $\logic{L}$ is a reflexive, monotone and transitive 
relation  $\logic{L}\subseteq\mathcal{P}(\alg{Fm})\times \alg{Fm}$, which is also invariant under substitutions, meaning that if $\langle \Gamma, 
\varphi\rangle\in\logic{L}$ then  $\langle h[\Gamma],h(\varphi)\rangle\in\logic{L}$ for every homomorphism $h:\alg{Fm}\to \alg{Fm}$. As usual, instead of using the infix notation, we denote a rule of the logic $\langle \Gamma,\varphi\rangle\in\logic{L}$  as $
\Gamma\vdash_{\logic{L}}\varphi$. We use $\psi\vdash\dashv_{\logic{L}}\varphi$ as a
 shorthand for $\varphi\vdash_{\logic{L}}\psi$ and $\psi\vdash_{\logic{L}}\varphi$.
  For $\varPhi\subseteq Fm$, by $\Gamma\vdash_{\logic{L}}\varPhi$ we understand $\Gamma\vdash_{\logic{L}}\varphi$, for each $\varphi\in\varPhi$. In this paper, we always assume that the logics we deal with are non trivial. Similar conventions apply when considering equations instead of formulas and the equational consequence $\vDash_{\class{K}}$ of a class of algebras $\class{K}$ instead of the consequence relation of a logic. A logic $\logic{L}$ is \emph{finitary} if $\Gamma\vdash_{\logic{L}}\varphi$ entails $\Delta\vdash_{\logic{L}}\varphi$, for some finite $\Delta\subseteq\Gamma$. In this paper, we will only deal with finitary logics, unless stated otherwise.
On the semantic side, our main tools will be provided by the theory of logical matrices. A matrix consists in a pair $\langle\alg{A}, F\rangle$ where $\alg{A}$ is an algebra and $F$ is a subset of its underlying set. Each matrix  $\langle\alg{A},F\rangle$ uniquely defines a logic $\vdash_{_{\langle\alg{A},F\rangle}}$ (in the same language as $\alg{A}$) as follows
\[\Gamma\vdash_{_{\langle\alg{A},F\rangle}}\varphi\iff \forall h:\alg{Fm}\to\alg{A},\text{ if }h[\Gamma]\in F \text{ then } h(\varphi)\in F,\]
for $\Gamma\cup\{\varphi\}\subseteq\alg{Fm}$.
 A matrix $\langle\alg{A},F\rangle$ is \emph{complete} or \emph{characteristic} for a logic $\logic{L}$ when 
 \[\Gamma\vdash_{\logic{L}}\varphi\iff\Gamma\vdash_{_{\langle\alg{A},F\rangle}}\varphi. \tag{Char}\label{tag char}\]
When the left-to-right implication of (\ref{tag char}) in the above display holds, we say that the matrix is  a \emph{model} of $\logic{L}$. Of course, this entails that $\alg{A}$ is an algebra of the same type of $\logic{L}$. In this case, $F$ is called a \emph{filter} of the logic $\logic{L}$. Intuitively, $F\subseteq\alg{A}$ is a $\logic{L}$-filter when it is closed under all the possibile interpretations of $\logic{L}$-rules over $\alg{A}$, and its elements are the so-called \emph{designated} values. A formula  $\varphi$ is a \emph{theorem} of $\logic{L}$ if $\vdash_{\logic{L}}\varphi$. A set of formulas $\Gamma$ is an \emph{antitheorem} of $\logic{L}$ if for each model $\langle\alg{A},F\rangle$ and $h:\alg{Fm}\to\alg{A}$ it holds $h(\gamma)\not\in F$ for some $\gamma\in\Gamma$. This is equivalent to saying that the formulas in $\Gamma$ cannot be jointly designated within a model of the logic. 
Each logic  $\logic{L}$ is canonically associated with a class of algebras, its algebraic counterpart $\Alg\logic{L}$. For instance, if we consider Classical logic ($\logic{CL}$), intuitionistic logic ($\logic{IL}$) and strong Kleene logic ($\logic{K}_{3}$), we obtain that $\Alg\logic{CL}$ is the class of Boolean algebras, $\Alg\logic{IL}$ is the class of Heyting algebras, and $\Alg\logic{K}_{3}$ is the class of Kleene lattices.  For any logic $\logic{L}$, the class $\Alg\logic{L}$ is closed under $\AlgI\AlgPS$, but need not to be closed under $\AlgS$. $\Alg\logic{L}$ is, intuitively, the class of algebras that support the ``interesting'' models of $\logic{L}$, the so-called reduced models. These are the models $\langle \alg{A},F\rangle$ such that  $\Tarski^{\alg{A}}F$ is the identity relation on $\alg{A}$, where the congruence $\Tarski^{\alg{A}}F$ on $\alg{A}$ is defined, for every $a,b\in \alg{A}$ as
\[\pair ab\in\Tarski^{\alg{A}}F\text{ if and only if }\varphi(a,\vec{c})\in G\iff\varphi(b,\vec{c})\in G,\]
for each formula $\varphi(x,\vec{z})$, each $\vec{c}\in \alg{A}$ and each $\logic{L}$-filter $G\supseteq F$ on $\alg{A}$. Crucially, $\alg{A}\in\Alg\logic{L}$ if and only if, for some $\logic{L}$-filter $F$ on $\alg{A}$, $\Tarski^{\alg{A}}F$ is the identity relation.

 We say that the matrix $\langle\alg{A},F\rangle$ is a submatrix of  $\langle\alg{B},G\rangle$ if $\alg{A}$ is a subalgebra of $\alg{B}$ (in symbols $\alg{A}\leq\alg{B}$) and $F=G\cap\alg{A}$. When $\mathcal{L}^{\prime}$ is a sublanguage of $\mathcal{L}$, given an $\mathcal{L}^{\prime}$-algebra $\alg{A}$ and an $\mathcal{L}$-algebra $\alg{B}$, we say that $\alg{A}$ is a \emph{subreduct} of $\alg{B}$ when $\alg{A}$ is a subalgebra of the algebra we obtain by restricting $\alg{B}$ to  the operations of $\mathcal{L}^{\prime}$. 
 When this condition holds and $\alg{A}$, $\alg{B}$ have the same universe, we say that $\alg{A}$ is the reduct of $\alg{B}$. The same terminology applies to matrices if we replace ``subalgebra'' with ``submatrix''. A class $\class{K}$ is a variety if it is closed under $\AlgH,\AlgS,\AlgP$, and it is a quasivariety if it is closed under $\AlgI\AlgS\AlgP\AlgPU$. When $\class{K}$ is a finite set of finite algebras, the smallest quasivariety containing $\class{K}$ is $\AlgI\AlgS\AlgP(\class{K})$. A quasivariety is axiomatized by a set of quasi-equations, namely universal first-order sentences in prenex normal form, where the non quantified part is an implication whose antecedent is a finite conjunction of equations and the consequent is a single equation.

 \section{ The external version of a three-valued logic}\label{Sec: external version}
   
We say that a logic  $\logic{L}$ is $n$-valued if $n$ is the smallest cardinality that a complete reduced model  for $\logic{L}$ can have. 
For instance, Classical logic $\logic{CL}$ is two-valued, Belnap-Dunn logic $\logic{FDE}$ is four-valued, while Intuitionistic logic is not $n$-valued, for any $n\in\omega$.
 Three-valued logics typically arise as an attempt to weakening classical logic and rejecting some of its distinctive inferences. This is done by expanding  the  classical (Boolean) domain $\{0,1\}$ with a third non-classical truth value ($\ant$), whose interpretation may vary depending on the specific context under consideration. For instance, in the Logic of Paradox ($\logic{LP}$) the rules of disjunctive syllogism ($\varphi,\neg\varphi\lor\psi\vdash\psi$) and of ex falso quodlibet ($\varphi\land\neg\varphi\vdash\psi$) admit a countermodel by evaluating $\varphi$ to $\ant$ and $\psi$ to $0$.
 
 From now on, it is convenient to denote the universe of a three-valued matrix by $\{0,\ant,1\}$, where $\{0,1\}$ is the subuniverse of the two elements Boolean algebra. This kind of logics  retain the classical language $\langle\land,\lor,\neg\rangle$ as a (not necessarily proper) fragment and, semantically, their operations behave classically when restricted to the Boolean universe.
Examples can be found in the realm of Kleene logics (e.g. $\logic{LP}, \logic{K}_{3}, \logic{B}, \logic{PWK}$), relevant logics (e.g. $\logic{RM}_{3}$, $\logic{S}_{3}$) and fuzzy logics ($\logic{\L}_{3},\logic{J}_{3}$), semantically defined in Example \ref{example: subclassical logics}.\footnote{We refer the interested reader to \cite{Avron1991,Ciucci,ciucci2014three,priest2008introduction} for a comprehensive treatment of the mathematical and philosophical aspects of three-valued logics.}
We condense these features within the definition of \emph{subclassical logic}\footnote{This expression is taken from \cite{ciuni2020normality}, where this name is used to identify similar logics.}.  From now on $\alg{B}_{2}$ denotes the two-element Boolean algebra with universe $\{0,1\}$.  
\begin{definition}\label{def: sublcassical logic}
Let $\logic{L}$ be a three-valued logic whose characteristic reduced matrix is $\langle\alg{A},F\rangle$ (which will be called \emph{defining matrix}). We say that $\logic{L}$ is a subclassical logic if 
\begin{enumerate}[(i)]
 \item $\alg{B}_{2}$ is a subreduct of $\alg{A}$ and $F\cap\alg{B}_{2}=\{1\}$;
\item $\alg{A}$ satisfies $\varphi\equals\neg\neg \varphi$.
\end{enumerate}
 
\end{definition}

 In general, no philosophical or arithmetical reading of the non-classical element $\ant$ should be assumed, as the appropriate interpretation only depends on its realization in a specific algebra. Notice that (ii) in Definition \ref{def: sublcassical logic} entails that $\ant=\neg\ant$  in every defining algebra of a subclassical logic. Thus, in every 
subclassical logic with defining matrix $\langle\alg{A},F\rangle$, the operation $\neg$ is uniquely determined  on $\alg{A}$ as $\neg 0=1$, $\neg 1=0$, $\neg\ant=\ant$.

\begin{figure}
 \begin{center}
  The algebra $\alg{SK}$ (type $\langle 2,2,1\rangle$)
\vspace{5 pt}

 \begin{tabular}{>{$}c<{$}|>{$}c<{$}>{$}c<{$}>{$}c<{$}}
   \land^{^{\alg{SK}}}& 0 & \ant & 1 \\[.2ex]
 \hline
       0 & 0& 0 & 0 \\
       \ant & 0 & \ant & \ant \\          
       1 & 0 & \ant & 1
\end{tabular}\hspace{10pt}
\begin{tabular}{>{$}c<{$}|>{$}c<{$}>{$}c<{$}>{$}c<{$}}
    \lor^{^{\alg{SK}}}& 0 & \ant & 1 \\[.2ex]
 \hline
       0 & 0 & \ant & 1 \\
       \ant & \ant & \ant & 1 \\          
       1 & 1 & 1 & 1
\end{tabular}

  \end{center}

 \begin{center}
 
 The algebra $\alg{WK}$ (type $\langle 2,2,1\rangle$)
  \vspace{5pt}

 \begin{tabular}{>{$}c<{$}|>{$}c<{$}>{$}c<{$}>{$}c<{$}}
   \land^{^{\alg{WK}}}& 0 & \ant & 1 \\[.2ex]
 \hline
       0 & 0 & \ant & 0 \\
       \ant & \ant & \ant & \ant \\          
       1 & 0 & \ant & 1
\end{tabular}\hspace{10pt}
\begin{tabular}{>{$}c<{$}|>{$}c<{$}>{$}c<{$}>{$}c<{$}}
    \lor^{^{\alg{WK}}}& 0 & \ant & 1 \\[.2ex]
 \hline
       0 & 0 & \ant & 1 \\
       \ant & \ant & \ant & \ant \\          
       1 & 1 & \ant & 1
\end{tabular}

 \end{center}
 
 \begin{center}
  The algebra $\alg{L}_{3}$ (type $\langle 2,2,2,1\rangle$)
\vspace{5pt}

 \begin{tabular}{>{$}c<{$}|>{$}c<{$}>{$}c<{$}>{$}c<{$}}
   \land^{^{\alg{L}_{3}}}& 0 & \ant & 1 \\[.2ex]
 \hline
       0 & 0& 0 & 0 \\
       \ant & 0 & \ant & \ant \\          
       1 & 0 & \ant & 1
\end{tabular}\hspace{10pt}
\begin{tabular}{>{$}c<{$}|>{$}c<{$}>{$}c<{$}>{$}c<{$}}
    \lor^{^{\alg{L}_{3}}}& 0 & \ant & 1 \\[.2ex]
 \hline
       0 & 0 & \ant & 1 \\
       \ant & \ant & \ant & 1 \\          
       1 & 1 & 1 & 1
\end{tabular}\hspace{10pt}
\begin{tabular}{>{$}c<{$}|>{$}c<{$}>{$}c<{$}>{$}c<{$}}
    \to^{^{\alg{L}_{3}}}& 0 & \ant & 1 \\[.2ex]
 \hline
       0 & 1 & 1 & 1 \\
       \ant & \ant & 1 & 1 \\          
       1 & 0 & \ant & 1
\end{tabular}

 \end{center}
 
 \begin{center}
  The algebra $\alg{S}_{3}$ (type $\langle 2,2,1\rangle$)
  \vspace{5pt}

 \begin{tabular}{>{$}c<{$}|>{$}c<{$}>{$}c<{$}>{$}c<{$}}
   \land^{^{\alg{S}_{3}}}& 0 & \ant & 1 \\[.2ex]
 \hline
       0 & 0 & 0 & 0 \\
       \ant & 0 & \ant & 1 \\          
       1 & 0 & 1 & 1
\end{tabular}\hspace{10pt}
\begin{tabular}{>{$}c<{$}|>{$}c<{$}>{$}c<{$}>{$}c<{$}}
    \lor^{^{\alg{S}_{3}}}& 0 & \ant & 1 \\[.2ex]
 \hline
       0 & 1 & 0 & 1 \\
       \ant & 0 & \ant & 1 \\          
       1 & 1 & 1 & 1
\end{tabular}

 \end{center}
 
 \begin{center}
  The algebra $\alg{Z}_{3}$ (type $\langle 2,2,2,1\rangle$)
  \vspace{5pt}

\begin{tabular}{>{$}c<{$}|>{$}c<{$}>{$}c<{$}>{$}c<{$}}
   \land^{^{\alg{Z}_{3}}}& 0 & \ant & 1 \\[.2ex]
 \hline
       0 & 0& 0 & 0 \\
       \ant & 0 & \ant & \ant \\          
       1 & 0 & \ant & 1
\end{tabular}\hspace{10pt}
\begin{tabular}{>{$}c<{$}|>{$}c<{$}>{$}c<{$}>{$}c<{$}}
    \lor^{^{\alg{Z}_{3}}}& 0 & \ant & 1 \\[.2ex]
 \hline
       0 & 0 & \ant & 1 \\
       \ant & \ant & \ant & 1 \\          
       1 & 1 & 1 & 1
\end{tabular}\hspace{10pt}
\begin{tabular}{>{$}c<{$}|>{$}c<{$}>{$}c<{$}>{$}c<{$}}
    \to^{^{\alg{Z}_{3}}}& 0 & \ant & 1 \\[.2ex]
 \hline
       0 & 1 & 1 & 1 \\
       \ant & 0 & \ant& 1 \\          
       1 & 0 & 0 & 1
\end{tabular}

 \end{center}
 \caption{Defining Algebras} \label{figure: algebras subclassical}
\end{figure}

  We now introduce a stock of  logics that will be instrumental for our purposes. These are paradigmatic representatives of subclassical logics (Definition \ref{def: sublcassical logic}) that the reader may keep in the background throughout the paper. This will provide a solid base to constantly test  the new concepts and the forthcoming results.
  
\begin{example}\label{example: subclassical logics} The following logics are subclassical (see Figure \ref{figure: algebras subclassical}).

\begin{itemize}
 \item $\logic{K}_{3}$, the logic defined by the matrix $\langle\alg{SK},\{1\}\rangle$. This is Strong Kleene logic, a well-known three-valued logic introduced in relation with  theories of truth and paradoxes. (See \cite{albuquerque2017algebraic} for a detailed formal study).

  \item  $\logic{LP}$, the logic defined by the matrix $\langle\alg{SK},\{1,\ant\}\rangle$. This is known as the Logic of Paradox, a famous paraconsistent logic introduced in \cite{asenjo1975logic} and further applied and developed by Priest \cite{Priestfirst}.
 \item $ \logic{B}$, the logic defined by the matrix $\langle\alg{WK},\{1\}\rangle$ This logic, known as Bochvar logic, has been introduced in \cite{Bochvar}. Recently, it has been studied in connection with logics of variable inclusion \cite{Bonziobook}  and their application to several topics in philosophy of logic \cite{CarraraEpistemic,SzmucEpistemic,Szmuctruth}.
 \item $\logic{PWK}$, the logic defined by the matrix $\langle\alg{WK},\{1,\ant\}\rangle$. Together with Bochvar logic, $\logic{PWK}$  belongs to the family of weak Kleene logics and it has been studied in the context of  a paraconsistency \cite{Hallden}. 
\item  $\logic{RM}_{3}$, the logic defined by the matrix $\langle\alg{Z}_{3},\{1,\ant\}\rangle$. This logic is the strongest extension of the logic of relevant mingle $\logic{RM}$, a well-known logic in the relevant family (see \cite{font1992note,Avron1991}). 
\item  $\logic{S}$, the logic defined by the matrix $\langle\alg{S}_{3},\{1,\ant\}\rangle$. Such system is equivalent to Soboci\'nski logic, a relatively well-known logic studied for its relation with linguistic fragments of the logic $\logic{RM}$ 
\cite{blok2004fragments,parks1972note,raftery2006equational}.
\item  $\logic{S}_{t}$, the logic defined by the matrix $\langle\alg{S}_{3},\{1\}\rangle$. To the best of our knowledge, this logic is not studied in the literature. 
\item $\logic{\L}_{3}$, the logic defined by the matrix $\langle\alg{L}_{3},\{1\}\rangle$. This system is widely recognized as \L ukasiewicz three-valued logic, and it probably is the most famous three-valued fuzzy logic (\cite{cignoli2007algebras, CiMuOt99}). 
\item $\logic{J}_{3}$, the logic defined by the matrix $\langle\alg{L}_{3},\{1,\ant\}\rangle$. This  logic has been developed by Da Costa and D'Ottaviano as the (weakly) paraconsistent counterpart of \L ukasiewicz logic. See \cite[Example 3.20]{Font16}.
\end{itemize}
 
\end{example}
 
The logics introduced in Example \ref{example: subclassical logics} do not form an exhaustive list of subclassical logics. Notice, moreover, that a subclassical logic is finitary, as it is complete with respect to a single finite matrix.
 Before proceeding, we clearly state an assumption that will be effective for the rest of the paper.
 \begin{assumption}\label{Assumption: Alg}
 We assume that, whenever $\logic{L}$ is a subclassical three-valued logic whose defining matrix is $\langle\alg{A},F\rangle$, then $\Alg\logic{L}=\AlgI\AlgS\AlgP(\alg{A})$.
\end{assumption}

All logics of Example \ref{example: subclassical logics}  satisfy Assumption 
 \ref{Assumption: Alg}, which always holds when $\logic{L}$ is algebraizable (see Theorem \ref{thm: czela on alg}). 
One can notice that the just mentioned three-valued subclassical logics can be partitioned into two 
 subclasses, depending of the set of designated values they preserve in the consequence relation. More precisely, given a three-valued algebra $\alg{A}$ having $\alg{B}_{2}$ as subreduct, the only two non-degenerate filters over $\alg{A}$ are $\{1\}$ and $\{1,\ant\}$. This fact is well-reflected by several logics discussed above, which are defined by matrices sharing the same algebraic reduct. The pairs $(\logic{K}_{3},
 \logic{LP})$, $(\logic{B},\logic{PWK})$, $(\logic{S}_{t},\logic{S})$, $(\logic{\L}_{3},\logic{J}_{3})$ are instances of this phenomenon with respect to the algebras $\alg{WK},\alg{SK},\alg{S}_{3},\alg{L}_{3}$.    
We are now ready to introduce the notion of external version of a subclassical logic.

\begin{definition}\label{def: external version 3-val logic}
Let $\logic{L}$ be a three-valued subclassical logic in the language $\mathcal{L}$, whose characteristic matrix is $\langle\alg{A},F\rangle$. The external version of $\logic{L}$, $\logic{L}^{e}$, is the logic defined by the matrix $\langle\alg{A}^{e},F\rangle$, where $\alg{A}^{e}$ is the algebra in the language $\mathcal{L}\cup\{\Delta_{1}\}$ with universe $A$ and whose operations are defined as:
$\circ^{\alg{A}^{e}}=\circ^{\alg{A}}$, for $\circ\in\mathcal{L}$ and
  \[
\Delta_{1}^{\alg{A}^{e}}(a)= \begin{cases}
1 \text{ if }  a=1\\
0 \text{ otherwise}.
\end{cases}
\]
\end{definition}

With these notions at play it is possible to see that $\logic{\L}_3$ is the external version of $\logic{K}_3$, $\logic{J}_3$ is the external version of $\logic{LP}$ and, of course, $\logic{B}^e$ (see \cite{bonzio2024bochvar}) is the external version of $\logic{B}$.

 It now seems natural to investigate the relationship between a subclassical logic $\logic{L}$ and its external version $\logic{L}^e$.  With this goal in mind, we proceed to we tackle two central questions. 
\begin{itemize}
\item How are the matrix models of $\logic{L}$ and of $\logic{L}^e$  related?
\item Which semantic properties are preserved or gained in the step from $\logic{L}$ to $\logic{L}^e$?
\end{itemize}

\section{The semantic meaning of external operators}\label{sec: semantic meaning}

The goal of this section is to study the semantic effect that external operators  have on the class of models naturally associated with a subclassical logic.  One might expect that, given a model of a subclassical logic, it is always possible to define an appropriate external operation $\Delta_{1}$, in order to obtain a model if its external version. However, this is not the case. The definability of such an operator forces different, non-trivial semantic conditions. More precisely, the central question to answer is: when can the algebraic reduct of a reduced model of  a subclassical logic $\logic{L}$ be expanded to a reduced model of its external version $\logic{L}^{e}$?
In the special case of the logic $\logic{K}_{3}$, a solution was given in the pioneering work of \cite{cignoli1965boolean}. In the case of Bochvar logic, this result has been recently established, with different methods, in \cite{bonzio2023structure,bonzio2024bochvar}. 
In both cases,  the conceptual framework employed rely on very peculiar properties of the respective algebras and of the associated logics. Here, as we shall se, we will provide a general solution to the problem which is based on the theory of \emph{nuclei}. As a result, it turns out that  the algebraic reduct of a reduced model $\alg{B}$ of $\logic{L}$ can be turned into a reduced model of $\logic{L}^{e}$ is possible if only if a specific conucleus is definable on $\alg{B}$.

As in the rest of the paper, let us denote by $\alg{A}$ the algebra that supports the reduced matrix $\langle \alg{A},F\rangle$ of a three-valued subclassical logic $\logic{L}$.   Also recall that, as stated in Assumption \ref{Assumption: Alg}, we require $\Alg\logic{L}=\AlgI\AlgS\AlgP(\alg{A})$ (notice that all the subclassical logics discussed so far fulfill this assumption). Now, if we consider $\alg{B}\in\AlgS\AlgP(\alg{A})$, namely $\alg{B}\leq\alg{A}^{I}$, for some set of indexes $I$,  each $b\in \alg{B}$ is a tuple $(\pi_{i}b)_{i\in I}$ such that $\pi_{i} b\in\{0,\ant,1\}$  for every $i\in I$, where $\pi_{i}b$ denotes the $i$-th entry of $b$. In plain words, each element of $\alg{B}$ is a tuple whose entries are values among the underlying set $\{0,\ant,1\}$ of $\alg{A}$. We now   introduce the notion of \emph{Boolean skeleton} for an algebra $\alg{B}\in\AlgS\AlgP(\alg{A})$, as it will play a central role in this section.
\begin{definition}\label{def: boolean skeleton 2}
Let $\alg{C}\leq\alg{A}^{I}$.  The Boolean Skeleton of $\alg{C}$, $\mathfrak{B}^{\alg{C}}$  for short, is defined as 
\[\mathfrak{B}^{\alg{C}}\assign\{c\in \alg{C}: \pi_{i}c\in\{0,1\}, \text{ for all }i\in I\}.\]
\end{definition}
One can see that, if $\mathfrak{B}^{\alg{C}}\neq\emptyset$ (the case $\mathfrak{B}^{\alg{C}}=\emptyset$ may arise, as we do not assume constant operations in our language), then it forms a subalgebra of $\alg{C}$. Moreover, for each $c\in\mathfrak{B}^{\alg{C}}$, $\pi_{i}(c\land \neg c)=0$, $\pi_{i}(c\lor\neg c)=1$ for each $i\in I$. In other words, the elements of the Boolean skeleton of $\alg{C}\leq\alg{A}^{I}$ are tuples whose values are among $\{0,1\}$ and they form a Boolean subalgebra of $\alg{C}$.
Since $\Alg\logic{L}$ is a quasivariety, this notion can be naturally extended to an arbitrary $\alg{B}\in\Alg\logic{L}$ by taking isomorphic pre-images, as specified by the next definition. 
\begin{definition}\label{def: boolean skeleton}
  Let $\alg{B}\in\Alg\logic{L}$, namely $\alg{B}\iso\alg{C}\leq\alg{A}^{I}$, where $h$ is an isomorphism $\alg{B}\to\alg{C}$. The Boolean skeleton of $\alg{B}$ determined by $h$ is $\mathfrak{B}^{\alg{B}}\assign h^{-1}[\mathfrak{B}^{\alg{C}}]$.
 
\end{definition}

 Notice that in Definition \ref{def: boolean skeleton} the Boolean skeleton of $\alg{B}$ is relative to a fixed isomorphism $h:\alg{B}\to \alg{C}$. Moreover, this definition has a purely technical role, and  it does not introduce any  real conceptual gain with respect to Definition \ref{def: boolean skeleton 2}. Indeed, up to isomorphism, all the algebras belonging to $\Alg\logic{L}$ can be thought as subalgebras of $\alg{A}^{I}$.

Now, let us define a partial order $\preceq$ on $\{0,\ant, 1\}$ as $0\preceq\ant\preceq 1$. For $a,b\in\alg{C}\leq\alg{A}^{I}$, set $a\preceq_{_{\alg{C}}} b$ if and only if $\pi_{i}a\preceq\pi_{i}b$, for each $i\in I$. When  $h:\alg{B}\to\alg{C}$ is an isomorphism and $c,d\in\alg{B}$, we set  $c\preceq_{_{\alg{B}}} d$ if and only if $h(c)\preceq_{_{\alg{C}}}h(d)$. It is trivial to check that $\preceq_{_{\alg{B}}}, \preceq_{_{\alg{C}}}$ are partial orders. It is clear that this order can be equally defined on algebras belonging to $\AlgI\AlgS\AlgP(\alg{A}^{e})$. 
The key tool which will allow to transparently describe the semantic meaning of external operators is the concept of $\emph{conucleus}$. The notion of conucleus found useful applications in logic, and  historically plays a central role in the study of varieties of residuated lattices, namely algebraic varieties arising as the equivalent algebraic semantics of substructural logics \cite{GaJiKoOn07}, \cite{galatos2005generalized}. 
For our purposes, a convenient formulation of the definition of conucleus is the following.

\begin{definition}\label{def: conucleus general}
 Let $\alg{B}\iso\alg{C}\leq\alg{A}^{I}$, and $h$ be an isomorphism $\alg{B}\to\alg{C}$.
 A  conucleus on $\alg{B}$ is a map $\sigma :\alg{B}\to{\alg{B}}$ which is monotone with respect to $\preceq_{_{\alg{B}}}$ and such that:
\begin{enumerate}[(i)]
 \item$\sigma (a)\preceq_{_{\alg{B}}} a$;
 \item $\sigma(a)\land\sigma(b)\preceq_{_{\alg{B}}} \sigma (a\land b)$;
 \item $\sigma (\sigma (a))=\sigma (a)$,
\end{enumerate}
for every $a,b\in\alg{B}$.

We will refer to $\sigma[\alg{B}]=\{a\in\alg{B}: \sigma(b)=a \text{ for some } a\in\alg{B}\}$ as the conuclear image of $\sigma$. As we shall see, we will be interested in special kinds of conuclei, whose features are defined below.
\end{definition}
\begin{definition}\label{def: conucleus}
 Let  $\alg{B}\iso\alg{C}\leq\alg{A}^{I}$, and $h$ be an isomorphism $\alg{B}\to\alg{C}$. A  conucleus $\sigma$ on $\alg{B}$ is called:
 
\begin{itemize}
 \item \emph{Boolean} if its conuclear image is $\mathfrak{B}^{\alg{B}}$;
 \item \emph{complete} if, for each  $a\in\alg{B}$,  $\bigvee^{\alg{A}^{I}}\{b\in\mathfrak{B}^{\alg{A}^{I}}: b\preceq_{_{\alg{A}^{I}}}h(a)\}\in h(\sigma[\alg{B}])$.
\end{itemize}
\end{definition}
We will say that $\alg{B}\in\Alg\logic{L}$ \emph{admits} a conucleus when there exists $\alg{C}\leq\alg{A}^{I}$ and an isomorphism $h:\alg{B}\to\alg{C}$ such that 
the conditions of Definition \ref{def: conucleus general} apply, and similarly for the specifications introduced in Definition \ref{def: conucleus}. 
Notice that if $\sigma$ is a Boolean conucleus on $\alg{B}$, then $\mathfrak{B}^{\alg{B}}\neq\emptyset$,  $\sigma[\alg{B}]=\mathfrak{B}^{\alg{B}}$ and $a\in\mathfrak{B}^{\alg{B}}$ entails $\sigma(a)=a$. Moreover, the condition that  $\sigma$ is complete entails that the join $\bigvee\{b\in\mathfrak{B}^{\alg{A}^{I}}: b\preceq_{_{\alg{A}^{I}}}h(a)\}$ exists in $\alg{A}^{I}$, for each $a\in\alg{B}$.

When needed, we will write $\sigma_{_{\alg{B}}}$ to emphasize the fact that this is conucleus on $\alg{B}$ and we will sometimes omit parenthesis in order to simplify the notation. If we replace $\alg{A}^{I}$ with $(\alg{A}^{e})^{I}$ and $\logic{L}$ with $\logic{L}^{e}$ in Definitions \ref{def: conucleus general},\ref{def: conucleus}, they obviously  extend to algebras belonging to $\Alg\logic{L}^{e}$.
The following fact is an immediate consequence of the definitions and the fact that the desired property of being a Boolean  conucleus is preserved by isomorphic pre images. 

\begin{fact}\label{fact: boolean elements invariant conuncleus}
Let $\alg{B}$ be the $\Delta_{1}$-free reduct of an $\Alg\logic{L}^{e}$-algebra $\alg{B}^{\Delta}$ and let  $h:\alg{B}^{\Delta}\to\alg{C}\leq(\alg{A}^{e})^{I}$ be an isomorphism. Suppose $\sigma_{_{\alg{C}}}$ is a  Boolean conucleus on $\alg{C}$.   The map $\sigma_{_{\alg{B}}}$ defined for every $a\in \alg{B}$ as $\sigma_{_{\alg{B}}}(a)=h^{-1}\sigma_{_{\alg{C}}}(h(a))$ is a Boolean conucleus on $\alg{B}$.
\end{fact}
The following lemma characterizes how the operation $\Delta_{1}$ is computed in a power of $\alg{A}^{e}$.
\begin{lemma}\label{lemma: delta is join of booleans}
Let $a\in(\alg{A}^{e})^{I}$. Then $\Delta_{1}a=\bigvee^{(\alg{A}^{e})^{I}}\{b\in\mathfrak{B}^{(\alg{A}^{e})^{I}}: b\preceq_{_{(\alg{A}^{e})^{I}}}a\}.$
\end{lemma}
\begin{proof}
Consider $a\in(\alg{A}^{e})^{I}$. In order to simplify the notation, let us write $\preceq$ as a shorthand for $\preceq_{_{(\alg{A}^{e})^{I}}}$ and set
\[X=\{b\in\mathfrak{B}^{(\alg{A}^{e})^{I}}: b\preceq a\}.\]
Notice that $X$ is clearly non-empty, as $\Delta_{1} a\in\mathfrak{B}^{(\alg{A}^{e})^{I}}$ and $\Delta_{1} a\preceq a$. This is true because  the operation $\Delta_{1}$ is computed component-wise  over a power of $\alg{A}^{e}$. Suppose, towards a contradiction, that  $b\in X$ and $b\npreceq\Delta_{1} a$. Thus $\pi_{i}b=1$ and $\pi_{i}\Delta_{1} a=0$ for some $i\in I$, which entails $\pi_{i}a=0$.   However,  $b\preceq a$ implies $\pi_{i}a=1$, a contradiction. This proves that $\Delta_{1} a$ is an upper bound of $X$. The fact that $\Delta_{1} a=\bigvee X$ follows from the fact that if $b\in X$ and $b\preceq \Delta_{1} a$ then $b=\Delta_{1} a$. 
\end{proof}

In the view of the above lemma, an equivalent description of how the value of the $\Delta_{1}$ operator  is computed is the following:
\[
\Delta_{1}a=\text{max}\{b\in\mathfrak{B}^{(\alg{A}^{e})^{I}}: b\preceq_{_{(\alg{A}^{e})^{I}}}a\}.\]

We are now ready to state the main result of this section. 
\begin{theorem}\label{thm: representation via conunclei}
Let $\alg{B}\in\Alg\logic{L}$. The following are equivalent:
\begin{enumerate}[(i)]
\item $\alg{B}$ can be equipped with the structure of an $\Alg\logic{L}^{e}$-algebra;
\item $\alg{B}$ admits a complete Boolean conucleus.
\end{enumerate}
 
\end{theorem}
\begin{proof}

 (i)$\Rightarrow$(ii).
 Let $\alg{B}\in\Alg\logic{L}$ and suppose it is possible to equip $\alg{B}$ with a unary operation $\Delta_{1}$, i.e. we can expand  $\alg{B}$ with a unary operation $\Delta_{1}$, so that the expanded algebra $\alg{B}^{\Delta}\in\Alg\logic{L}^{e}$. This means that $\alg{B}^{\Delta}$ is isomorphic via a map $h$ to some $\alg{C}\leq(\alg{A}^{e})^{I}$. Thus, the  elements of $\alg{C}$ are tuples whose entries are among $\{0,\ant, 1\}$, i.e., for $a\in\alg{C}$,  $\pi_{i}c\in\{0,\ant, 1\}$ for every $i\in I$.   Since $\alg{C}\leq(\alg{A}^{e})^{I}$, $\Delta_{1}^{\alg{C}}$ is map $\alg{C}\to\mathfrak{B}^{\alg{C}}$. This is true because, for any element $a\in\alg{C}$, $\Delta_{1}^{\alg{C}}a=(\Delta^{(\alg{A}^{e})}\pi_{i}a)_{i\in I}$ and clearly $\Delta_{1}^{(\alg{A}^{e})}\pi_{i}a\in\{0,1\}$ for every $i\in I$. Thus, $\Delta_{1}^{\alg{C}}$ is a map $\alg{C}\to\mathfrak{B}^{\alg{B}}$. We now show that $\Delta_{1}^{\alg{C}}$ is a Boolean conucleus on $\alg{C}$. To this end, recall the partial order $0\preceq\ant\preceq 1$ and its realization $\preceq_{_{\alg{C}}}$  over the algebra $\alg{C}$ defined, for every $a,b\in\alg{C}$ as $a\preceq_{_{\alg{C}}}b$ if and only if $\pi_{i}a\preceq \pi_{i}b$ for every $i\in I$. If we consider $a\in\alg{C}$, it holds that
  \[
\pi_{i}\Delta_{1}^{\alg{C}}a= \begin{cases}
1 \text{ if }  \pi_{i}a=1\\
0 \text{ otherwise},
\end{cases}\tag{Def $\Delta_{1}^{\alg{C}}$}\label{tag delta1}
\]
 thus $\Delta_{1}^{\alg{C}}a\preceq_{_{\alg{C}}}a$. Moreover, if $a\in\mathfrak{B}^{\alg{C}}$, clearly $\Delta_{1}^{\alg{C}}a=a$.
 
 Let  now $a\preceq_{_{\alg{C}}} b$.  By the above display (\ref{tag delta1}) it is immediate to check that that if $\pi_{i}a\preceq_{_{\alg{C}}}\pi_{i}b$ then $\pi_{i}\Delta_{1}^{\alg{C}}a\preceq_{_{\alg{C}}}\pi_{i}\Delta_{1}^{c}b$, so $\Delta_{1}^{\alg{C}}$ is monotone with respect to $\preceq_{_{\alg{C}}}$. The fact that $\Delta_{1}^{\alg{C}}(\Delta_{1}^{\alg{C}}a)=\Delta_{1}^{\alg{C}}a$ can be checked similarly. 
 We now prove (ii) in Definition \ref{def: conucleus}, namely that $\Delta_{1}^{\alg{C}}a\land\Delta_{1}^{\alg{C}}b\preceq_{_{\alg{C}}}\Delta_{1}^{\alg{C}}(a\land b)$. Observe that $\Delta_{1}^{\alg{C}}a\land \Delta_{1}^{\alg{C}}b, \Delta_{1}^{\alg{C}}(a\land b)\in\mathfrak{B}^{\alg{C}}$ and that, for every $c,d\in \mathfrak{B}^{\alg{C}}$ it holds that $c \preceq_{_{\alg{C}}} d$ if and only if $c\land d= c$. Moreover, as $\alg{C}\in\AlgP(\alg{A}^{e})$, it satisfies all the quasi-equations that hold in $\alg{A}^{e}$. Since it is easy to check that
 \[\alg{A}^{e}\vDash (\Delta_{1}\varphi\land\Delta_{1}\psi)\land(\Delta_{1}(\varphi\land\psi))\equals(\Delta_{1}\varphi\land\Delta_{1}\psi),\] then
 \[ (\Delta_{1}a\land\Delta_{1}b)\land(\Delta_{1}(a\land b))=(\Delta_{1}a\land\Delta_{1}b)\]
 holds in $\alg{C}$, thus $\Delta_{1}^{\alg{C}}a\land\Delta_{1}^{\alg{C}}b\preceq_{_{\alg{C}}}\Delta_{1}^{\alg{C}}(a\land b)$. This shows $\Delta_{1}^{\alg{C}}$ is a Boolean conucleus. By Fact \ref{fact: boolean elements invariant conuncleus} the map $\sigma_{_{\alg{B}}}$ defined for every $a\in \alg{B}$ as $\sigma_{_{\alg{B}}}(a)=h^{-1}\sigma_{_{\alg{C}}}(h(a))$ is a  Boolean conucleus on $\alg{B}$. It remains to show this conucleus is also complete, but this follows from Lemma \ref{lemma: delta is join of booleans}, upon noticing that, for $a\in\alg{B}$, $\Delta^{(\alg{A}^{e})^{I}}h(a)=\Delta_{1}^{\alg{C}}h(a)\in\mathfrak{B}^{\alg{C}}= h(\sigma_{_{\alg{B}}}[\alg{B}])$.

(ii)$\Rightarrow$(i). Consider $\alg{B}\in\Alg\logic{L}=\AlgI\AlgS\AlgP(\alg{A})$.  Suppose $
\alg{B}$ admits a complete Boolean conucleus $\sigma_{_{\alg{B}}}:\alg{B}\to\mathfrak{B}
^{\alg{B}}$, where $\mathfrak{B}^{\alg{B}}$ is the Boolean skeleton 
determined by some algebra $\alg{C}\leq\alg{A}^{I}$ isomorphic to $\alg{B}$ via $h:\alg{B}
\to\alg{C}$. It is immediate to check that the map defined for every $a\in\alg{C}$ as $\sigma_{_{\alg{C}}}a=h(\sigma_{_{\alg{B}}}h^{-1}a)$ is a  Boolean conucleus on $\alg{C}$.  Consider the algebra $\alg{B}^{\sigma}$ obtained by expanding $\alg{B}$ with the operation $\sigma_{_{\alg{B}}}$, and similarly for $\alg{C}$ with respect to $\sigma_{_{\alg{C}}}$. Clearly the 
isomorphism $h$ extends to an isomorphism between the expanded algebras $h:\alg{B}^{\sigma}\to\alg{C}^{\sigma}$ because, for  $a\in\alg{B}$, $\sigma_{_{\alg{C}}} h(a)=h(\sigma_{_{\alg{B}}} h^{-1}ha)=h(\sigma_{_{\alg{B}}}a)$.  We only need to show that, for $h(a)\in\alg{C}^{\sigma}$, ${\sigma_{_{\alg{C}}}}h(a)=\Delta_{1}h(a)$ where $\Delta_{1}$ is the operation computed on $(\alg{A}^{e})^{I}$. From the fact that $\sigma_{_{\alg{B}}}$ is complete,  together with Lemma \ref{lemma: delta is join of booleans}, we obtain that
\[\Delta_{1}h(a)=\bigvee\{b\in\mathfrak{B}^{(\alg{A}^{e})^{I}}: b\preceq h(a)\}\in h(\sigma_{_{\alg{B}}}[\alg{B}])=\mathfrak{B}^{\alg{C}},\]
where the join takes place in $(\alg{A}^{e})^{I}$ and $\preceq$ is a shorthand for $\preceq_{_{(\alg{A}^{e})^{I}}}$.
Clearly, ${\sigma_{_{\alg{C}}}}h(a)\preceq h(a)$ because ${\sigma_{_{\alg{C}}}}$ is a 
Boolean conucleus, so by the above display $\sigma_{_{\alg{C}}}h(a)\preceq\Delta_{1}h(a)$. Moreover, since $\sigma_{_{\alg{C}}}$ is monotone, $\Delta_{1}h(a)\preceq h(a)$ entails $\sigma_{_{\alg{C}}}\Delta_{1}h(a)=\Delta_{1}h(a)
\preceq\sigma_{_{\alg{C}}}h(a)$, thus $\sigma_{_{\alg{C}}}h(a)=\Delta_{1}h(a)$, as desired.
\end{proof}

The intuitive meaning of being a complete Boolean conucleus can be described as follows. If we see the elements of a model as truth-values, it expresses the fact that, whenever $a$ is a non-classical truth value, then it is  possible to select the greatest Boolean truth value below $a$. In other words, completeness ensures that a notion of ``closest'' Boolean value is always defined in a model.

The following corollary provide a valuable criterion to identify the algebras of $\Alg\logic{L}$ that cannot be equipped with the structure of an $\logic{L}^{e}$-algebra.
\begin{corollary}\label{cor: on separting}
Suppose $\alg{B}\in\Alg\logic{L}$ admits a Boolean complete conucleus $\sigma$. Then, for  different elements $a,b\in\alg{B}$ it holds that
\[ \text{if }\sigma(a)=\sigma(b)\text{ then } \sigma(\neg a)\neq\sigma(\neg b).\tag{Sep}\label{quasieq sep}\]
\end{corollary}
\begin{proof}
 Theorem \ref{thm: representation via conunclei} ensures that the expansion $\alg{B}^{\sigma}$ of $\alg{B}$ is isomorphic via $h$ to some $\alg{C}\leq(\alg{A}^{e})^{I}$, thus, for $a\in\alg{B}$, $h(\sigma a)=\Delta_{1}h(a)$, where $\Delta_{1}$ is the operation on $(\alg{A}^{e})^{I}$. It easy to check that 
  \[\alg{A}^{e}\vDash\Delta_{1}\varphi\thickapprox\Delta_{1}\psi ~ \& ~ \Delta_{1}\neg\varphi\thickapprox\Delta_{1}\neg\psi\implies\varphi\thickapprox\psi.\]
  Since the validity of quasi-equations persists in isomorphic preimages of subalgebras of products, (\ref{quasieq sep}) is true in $\alg{B}^{\sigma}$. Thus, for different $a,b\in\alg{B}$, $\sigma(a)=\sigma(b)$ entails $\sigma(\neg a)\neq\sigma(\neg b)$.
 \end{proof}
The concrete meaning of this result is  briefly exemplified below.
\begin{example}
Consider the four element Kleene lattice depicted in Figure \ref{4 element chain}. In the light of Corollary \ref{cor: on separting}, it is easy to verify that and why this algebra cannot be turned into a three-valued \L ukasiwicz  algebra. To see this, suppose it admits a complete Boolean conucleus $\sigma$. This entails that, the Boolean skeleton determined by an isomorphism  $h:\alg{B}\to\alg{C}\leq\alg{A}^{I}$ is not empty. Therefore, as it can be easily checked, this  Boolean skeleton consists of the elements $\{0,1\}$, depicted with  solid circles as in Figure \ref{4 element chain}. The induced Boolean conucleus is represented by arrows. Clearly, $\sigma$  does not satisfy the condition  (\ref{quasieq sep}), as $a\neq\neg a$ but $\sigma (a)=\sigma(\neg a)$, against Corollary \ref{cor: on separting}.
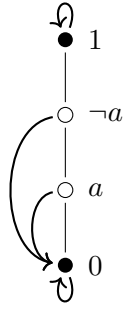
\begin{figure}
 \begin{center}
\begin{tikzpicture}[scale=1, dot/.style={inner sep=2.5pt,outer sep=2.5pt}, solid/.style={circle,fill,inner sep=2pt,outer sep=2pt}, empty/.style={circle,draw,inner sep=2pt,outer sep=2pt}]

\node  [label={right:$0$}](0) at (1,4) [solid] {};
 \node  [label={right:$a$}] (a) at (1,5) [empty] {};
  \node  [label={right:$\neg a$}] (-a) at (1,6) [empty] {};
  \node  [label={right:$1$}] (1) at (1,7) [solid] {};

 \draw[thick,->][ bend right=80 ] (a) edge (0);
 \draw[thick,->][  ] (0) edge [in=-110, out=-70 ,looseness=10](0);
 \draw[thick, ->][ ] (1) edge [in=110, out=70 ,looseness=10] (1);
 \draw[thick,->][ bend right=80 ] (-a) edge (0);
 \draw[-] (0) edge (a);
 \draw[-] (a) edge (-a);
 \draw[-] (-a) edge (1);
\end{tikzpicture}
\caption{The four element chain Kleene lattice $\alg{K}_{4}$}\label{4 element chain}
\end{center}
\end{figure}
\end{example}

For a wide class of subclassical logics the above results can be refined, and the extra-condition imposed on a conucleus of being complete can be substituted by a single inequality.  In such cases, the sufficient and necessary conditions needed to turn an algebra $\alg{B}\in\Alg\logic{L}$ into an $\Alg\logic{L}^{e}$-algebra can be verified on $\alg{B}$, without any reference to the second condition in Definition \ref{def: conucleus}. If we refer to the stock of logics introduced in Example \ref{example: subclassical logics}, this is the case of $\logic{K}_{3},\logic{LP},\logic{B},\logic{PWK},\logic{RM}_{3}$.
\begin{lemma}[\cite{burris1981course}, Thm. 7.15]\label{lem: intersezione kernel}
 Let $\alg{A},\alg{B}$ be similar algebras. Then $\alg{A}\in\AlgI\AlgS\AlgP(\alg{B})$ if and only if there exists a family of homomorphisms $h_{i}:\alg{A}\to\alg{B}$ such that $\bigcap_{i\in I}\ker h_{i}$ is the identity relation.
\end{lemma}
\begin{theorem}\label{thm: conucleus for kleene}
 Let $\logic{L}$ be a subclassical logic such that  $ 0\lor\ant\neq 0$ is true in $\alg{A}^{e}$. Then, for $\alg{B}\in\Alg\logic{L}$ the following are equivalent:
 
\begin{enumerate}[(i)]
 \item $\alg{B}$ can be turned into an $\Alg\logic{L}^{e}$-algebra;
 \item there exists a Boolean conucleus $\sigma$ on $\alg{B}$ satisfying $x\preceq\sigma x\lor\neg x$.
\end{enumerate}
\end{theorem}
\begin{proof}
 (i)$\Rightarrow$(ii). 
 If $\alg{B}$ can be turned into an $\Alg\logic{L}^{e}$-algebra, then  the expansion $\alg{B}^{\Delta}$ of $\alg{B}$ is isomorphic  to $\alg{C}\leq(\alg{A}^{e})^{I}$ via $h$. Theorem \ref{thm: representation via conunclei} ensures that $\Delta_{1}$ on $\alg{B}$ is a Boolean conucleus. The fact that for $a\in\alg{B}$, it holds $a\preceq\Delta_{1} a\lor\neg a$ is equivalent to  $h(a)\preceq\Delta_{1} h(a)\lor\neg h(a)$ on $\alg{C}$. To prove this, it suffices to  mechanically check that, for each $i\in I$, $\pi_{i}h(a)\preceq\pi_{i}(\Delta_{1}(h (a)))\lor \pi_{i}(\neg h(a))$, upon recalling the assumption that $\ant\lor 0\neq 0$ holds in $\alg{A}^{e}$. 
 
(ii)$\Rightarrow$(i) Let $\alg{B}$ as in the statement, thus, as in Theorem \ref{thm: representation via conunclei}, $\alg{B}^{\sigma}\iso\alg{C}^{\sigma}$ via $h$, for some $\alg{C}\leq\alg{A}^{I}$. For each $i\in I$ consider the projection map $\pi_{i}:\alg{C}\to\alg{A}^{e}$. It is clear that each $\pi_{i}$ preserve the operations of the language of $\alg{A}$. We now show that each map $\pi_{i}$ extends to a homomorphism $\alg{C}^{\sigma}\to\alg{A}^{e}$. Thus, we need to prove that, for each $a\in\alg{C}$ it holds $\pi_{i}(\sigma_{_{\alg{C}}} a)=\Delta_{1}(\pi_{i}a)$. To this end, we consider three exhaustive and mutually exclusive cases. If $\pi_{i}a=0$ or $\pi_{i}a=\ant$ then, since $\sigma_{_{\alg{C}}}a\leq a$ and $\sigma_{_{\alg{C}}}a\in\mathfrak{B}^{\alg{C}^{\sigma}}$, necessarily $\pi_{i}(\sigma_{_{\alg{C}}}a)=0=\Delta_{1}(\pi_{i}a)$. Suppose now $\pi_{i}a=1$. Since $a\preceq\sigma_{_{\alg{C}}}a\lor\neg a$ is true in $\alg{C}^{\sigma}$, then $\pi_{i}(\sigma_{_{\alg{C}}}a\lor\neg a)=1$. Thus, since  $\pi_{i}\neg a=0$, we conclude $\pi_{i}(\sigma_{_{\alg{C}}} a)=1=\Delta_{1}(\pi_{i}a)$, as desired. Moreover notice that
$\bigcap_{i\in I}\Ker\pi_{i}$ is the identity relation on $\alg{B}$,
so $\alg{B}^{\sigma}\in\AlgI\AlgS\AlgP(\alg{A}^{e})=\Alg\logic{L}^{e}$ by Lemma \ref{lem: intersezione kernel}.
 \end{proof}
 
 \section{External operators and classical recapture}\label{sec: properties external}

 As already hinted, the operator $\Delta_{1}$ can be seen a way to provide a yes/no answer, understood as a $0/1$-answer, to the question of whether its argument is both classical and true. In other words, $\Delta_{1}$ enriches the expressive power of the underlying logic $\logic{L}$ in order to re-capture the ``external'' notion of classical truth, which cannot be represented by the linguistic resources of $\logic{L}$ alone. This section explores some of the classical features that are recaptured in the step from a sublassical logic $\logic{L}$ to its external version $\logic{L}^{e}$.

 Informally speaking, it is easy to see that if $\Gamma\vdash\varphi$ is a classical rule failing in $\logic{L}$, the external version of $\logic{L}$ validates a rule  encoding the idea that ``whenever the arguments of the formulas occurring in the rule can be safely asserted, the rule can be safely applied". The next proposition shows how to produce $\logic{L}^{e}$-valid inferences of this kind.  
\begin{proposition}\label{fact: inference recapture}
Let $\gamma_{1},\dots,\gamma_{n}\vdash\varphi$ be a classically valid rule and $
\alg{Fm}$ be the formula algebra in the language of $\logic{L}^{e}$. Then 
$h(\gamma_{1}),\dots, h(\gamma_{n})\vdash h(\varphi)$ is a rule of $\logic{L}^{e}$ where
   $h:\alg{Fm}\to\alg{Fm}$ is the substitution mapping $x\mapsto\Delta_{1}x$ for each variable $x$.
\end{proposition}
\begin{proof}
 Let $\gamma_{1},\dots,\gamma_{n}\vdash\varphi$ be a classically valid rule and let $h$ be as in the statement.  The restriction of $h$ to the language  $\langle\land,\lor,\neg,\Delta_{1}\rangle$ is a substitution on the formula algebra of $\logic{CL}^{e}$, thus, $h(\gamma_{1}),\dots, h(\gamma_{n})\vdash h(\varphi)$ is a rule of $\logic{CL}^{e}$. Let $g:\alg{Fm}\to\alg{A}^{e}$, where $\langle\alg{A},F\rangle$ is the defining matrix of $\logic{L}$, a homomorphism such that $g(h(\gamma_{1}))\in F,\dots, g(h(\gamma_{n}))\in F$. By definition of $h$ this entails $g(h(\gamma_{i}))=1$ for each $1\leq i\leq n$.  Since $h(\gamma_{1}),\dots, h(\gamma_{n})\vdash h(\varphi)$ is a rule of $\logic{CL}^{e}$ we conclude that $g(h(\varphi))=1$. 
\end{proof}

Let us briefly provide an example.
\begin{example}
 Consider the logic $\logic{PWK}^{e}$. Clearly the rule of modus ponens $\varphi,\neg\varphi\lor\psi\vdash\psi$ fails in $\logic{PWK}$. From Proposition \ref{fact: inference recapture} it follows that \[\Delta_{1}\varphi,\neg(\Delta_{1}\varphi)\lor\Delta_{1}\psi\vdash\Delta_{1}\psi \]
 is a valid rule of $\logic{PWK}^{e}$.
 A similar argument shows that $\Delta_{1}\varphi\lor\neg\Delta_{1}\varphi$ is a theorem of $\logic{B}^{e}$. It is important to notice that  $\Delta_{1}(\varphi\lor\neg\varphi)$ is not a theorem of $\logic{B}^{e}$: this emphasizes the role of the substitution $h$ in Proposition \ref{fact: inference recapture}.
\end{example}

We now  recall the notion of algebraizability. This concept lies at the core of algebraic logic, and it requires the availability of a mutual interpretation between a  logic $\logic{L}$ and the equational consequence of a specific class of algebras $\class{K}$. This interpretation is witnessed by two maps $\boldtau(x)$, $\boldrho(x,y)$ that assign a set of equations $\boldtau(\varphi)$ to each formula $\varphi$ and a set of formulas $\boldrho(\varphi,\psi)$ to each equation $\varphi\equals\psi$.

\begin{definition}\label{def: algebraizability}
 A logic $\logic{L}$ is algebraizable with respect to the class of algebras $\class{K}$ (in the same language as $\logic{L}$) if there are maps $\boldtau:Fm\to\mathcal{P}(Eq)$ from formulas to sets of equations and $\boldrho:Eq\to\mathcal{P}(Fm)$ from equations to sets of formulas such that the following hold:
\begin{gather*}
 \Gamma\vdash_{\logic{L}}\varphi\iff\boldtau(\Gamma)\vDash_{\class{K}}\boldtau(\varphi) \tag{Alg1} \label{ALG1}\\
   \varphi\equals\psi\Dashv\vDash_{\class{K}}\boldtau\boldrho(\varphi\equals\psi)\tag{Alg2} \label{ALG2}.\\
\end{gather*}
\end{definition}
In the case of classical logic the  two maps witnessing the algebraizability with respect to the variety of Boolean algebras $\class{BA}$ are defined as $\boldtau(\varphi)=\{\varphi\equals\varphi\lor\neg\varphi\}$ and $\boldrho(\varphi\equals\psi)=\{\varphi\to\psi,\psi\to\varphi\}$.  Conditions \ref{ALG1}, \ref{ALG2} are instantiated accordingly by letting $\class{K}=\class{BA}$.
Remarkably, if a logic is algebraizable with respect to $\class{K}$, then it is algebraizable with respect to $\Alg\logic{L}$.  In this case $\Alg\logic{L}$ is called the equivalent algebraic semantics of $\logic{L}$. 
The following important theorem will be used in the forthcoming part. 

\begin{theorem}\cite[Thm.3.2.2]{Cz01}\label{thm: czela on alg}. 	
 Let  $\logic{L}$ be an algebraizable with respect to a finite set of formulas  $\boldrho(x,y)$. If $\logic{L}$  has a reduced characteristic matrix $\langle\alg{A},F\rangle$ and $\alg{A}$ is finite, then $\Alg\logic{L}=\AlgI\AlgS\AlgP(\alg{A})$. \end{theorem}
 
In general, the notion of algebraizability states an equivalence, witnessed by the maps $\boldtau,\boldrho$, between the consequence relation of a logic and the equational consequence of its algebraic counterpart. 

When a logic $\logic{L}^{e}$ has $\langle\alg{A}^{e},F\rangle$ as defining matrix, it is possible to define two more unary operators of special interest, as follows:
\[\Delta_{0}\varphi\assign\Delta_{1}\neg\varphi \text{ and } \Delta_{\ant}\varphi\assign\neg(\Delta_{1}\varphi\lor\Delta_{1}\neg\varphi).\]

They are computed on $\alg{A}^{e}$ according with the truth tables below
\begin{center}
 \begin{tabular}{>{$}c<{$}|>{$}c<{$}>{$}c<{$}>{$}c<{$}}
   \varphi&   \Delta_{0}\varphi \\[.2ex]
 \hline
       0 & 1 \\
       \ant & 0 \\          
       1 & 0
\end{tabular}\hspace{10pt}
\begin{tabular}{>{$}c<{$}|>{$}c<{$}>{$}c<{$}>{$}c<{$}}
   \varphi&\Delta_{\ant}\varphi  \\[.2ex]
 \hline
       0 & 0 \\
       \ant & 1 \\          
       1 & 0
\end{tabular}
\end{center}

The intended  reading is clear: $\Delta_{0}\varphi$ intuitively asserts that ``$\varphi$ is classical and false'', while $\Delta_{\ant}\varphi$ asserts that ``$\varphi$ is non-classical''.
It is now convenient to fix new notations. Let $\alg{A}$ be an arbitrary but fixed three-element algebra having $\alg{B}_{2}$ as subreduct. By $\logic{L}_{1}$ we denote the logic defined by $\langle\alg{A},\{1\}\rangle$, while $\logic{L}_{\ant}$  is the logic defined by $\langle\alg{A},\{1,\ant\}\rangle$. 
This also applies to the external version of a subclassical logic. Thus,  $\logic{L}_{1}^{e}$ is the logic defined by the matrix $\langle\alg{A}^{e},\{1\}\rangle$ while $\logic{L}_{\ant}^{e}$ is the logic defined by $\langle\alg{A}^{e},\{1,\ant\}\rangle$. For a set $\Gamma\subseteq Fm$ and $i\in\{0,\ant,1\}$ we denote $\{\Delta_{i}(\gamma)\}_{\gamma\in\Gamma}$ as $\Delta_{i}\Gamma$.

With this new information at hand, it is easy to show that $\logic{L}^{e}$ is an algebraizable logic.

\begin{theorem}\label{thm: external version is algebraizable}
 Let $\logic{L}$ be a  subclassical three-valued logic, and $\langle\alg{A}^{e},F\rangle$ be the characteristic matrix of $\logic{L}^{e}$. Then $\logic{L}^{e}$ is algebraizable, and its  equivalent algebraic semantics is $\Alg\logic{L}^{e}=\AlgI\AlgS\AlgP(\alg{A}^{e})$.
 \end{theorem}
 
\begin{proof}
Define a new binary operation  $\to$  as follows:
\[\varphi\to\psi=\Delta_{0}\varphi\lor(\Delta_{\ant}\varphi\land(\Delta_{\ant}\psi\lor\Delta_{1}\psi))\lor(\Delta_{1}\varphi\land\Delta_{1}\psi).\]

 If  $\logic{L}^{e}=\logic{L}^{e}_{1}$, consider $\boldtau_{1}(\varphi)\assign\{\varphi\equals  \varphi\to\varphi\}$. 
If $\logic{L}^{e}=\logic{L}^{e}_{\ant}$, consider $\boldtau_{\ant}(\varphi)\assign\{\Delta_{1}\varphi\lor\Delta_{\ant}\varphi\equals 1\}$. Upon defining $\boldrho(\varphi\equals\psi)\assign\{\varphi\to\psi,\psi\to\varphi\}$, it is immediate to check that  $\boldtau_{1},\boldrho$ witnesses the algebraizability of $\logic{L}^{e}_{1}$ and $\boldtau_{\ant},\boldrho$ the algebraizability of $\logic{L}^{e}_{\ant}$, both with respect to $\alg{A}^{e}$.  Finally, the fact that the equivalent algebraic semantics of $\logic{L}^{e}$ is $\AlgI\AlgS\AlgP(\alg{A}^{e})$ follows by Theorem \ref{thm: czela on alg}.
 \end{proof}
 
 Of course, the algebraizability of some external versions of a logic, such as  $\logic{\L}_{3}$, is a well-known fact. In this case, the usual way of interpreting equations into formulas differs from the one encoded by $\boldrho$: the standard transformer consists of the set $\{\varphi\to_{\logic{\L}}\psi,\psi\to_{\logic{\L}}\varphi\}$, where $\to_{\logic{\L}}$ is the so-called \L ukasiewicz implication, defined according to the truth tables in Figure \ref{figure: algebras subclassical}.
Clearly the two operations $\to_{\logic{\L}}$ and $\to$ are different, as $\to$ only takes values among $0,1$. However, as required by the general theory of algebraizability, it holds:
\[\varphi\to\psi,\psi\to\varphi\vdash\dashv_{\logic{\L}_{3}}\varphi\to_{\logic{\L}}\psi, \psi\to_{\logic{\L}}\varphi.\]

 Recall that a logic $\logic{L}$ has the Deduction Theorem (DDT) if there is a set of formulas $\ddtset(x, y)$ such that
\begin{align*}
  \Gamma, \varphi \vdash_{\logic{L}} \psi \iff \Gamma \vdash_{\logic{L}} \ddtset(\varphi, \psi). \tag{DDT}\label{DDT}
\end{align*}
Notice also that being algebraizable and having the DDT are independent properties.
The DDT is lost in several subclassical three-valued logics considered so far, such as $\logic{B}$, $\logic{PWK},$  $\logic{K}_{3}$ and $\logic{LP}$.  Differently, other subclassical three-valued logics enjoy the DDT: this is the case, for instance, of $\logic{\L}_{3}, \logic{J}_{3}, \logic{S}$, and $\logic{RM}_{3}$. 
Unsurprisingly, the Deduction Theorem is a further property recaptured in the step from $\logic{L}$ to its external version. The next fact provides the details.

\begin{fact}\label{thm: ddt in external versions}
The external version of a three-valued subclassical logic has the DDT.
\end{fact}
\begin{proof}
The set $\varphi\to_{1}\psi=\neg\Delta_{1}\varphi\lor\Delta_{1}\psi$ witnesses the DDT  for $\logic{L}^{e}_{1}$, 
and $\varphi\to_{\ant}\psi=\neg(\Delta_{1}\varphi\lor\Delta_{\ant}\varphi)\lor(\Delta_{1}\psi\lor\Delta_{\ant}\psi)$ does it so for for $\logic{L}^{e}_{\ant}$. \end{proof}

Notice that  the traditional DDT-set for $\logic{\L}_{3}$ is a singleton containing the formula $x\to_{\logic{\L}}(x\to_{\logic{\L}}y)$, which can be equivalently written as $\neg\Delta_{1}x\lor y$. A quick computation shows that $x\to_{1}y$ is a different operation on $\alg{L}_{3}$ from $x\to_{\logic{\L}} (x\to_{\logic{\L}}y)$, as $1\to_{\logic{\L}} (1\to_{\logic{\L}}\ant)=\ant$, while $1\to_{1}\ant=0$. Thus, as for the notion of algebraizability,  there is a term witnessing the DDT in $\logic{\L}_{3}$ where the  non-classical value $\ant$ never appears in the output when interpreted in $\alg{L}_{3}$.

If we take another look at the definition of algebraizability, we can see that it can be generalized to relate two logics rather than just relating a logic to the equational consequence of a class of algebras. 
This generalization requires an appropriate modification of the maps  $\boldtau,\boldrho$ and leads to the concept of \emph{deductive equivalence}.\footnote{For our purposes, it is inconvenient  to present the details of the beautiful theory underlying this approach, which can be found in \cite{blok2006equivalence} and \cite{galatos2009equivalence}.} 
\begin{definition}\label{def: deductive equivalence}
Let $\logic{L}$ be a logic in the language $\mathcal{L}$  and  $\logic{L}^{\prime}$ be a logic in the language $ \mathcal{L}^{\prime}$. $\logic{L}$  and $\logic{L}^{\prime}$ are deductively equivalent when there exist maps $\boldtau:\alg{Fm}_{\mathcal{L}}\to\mathcal{P}(\alg{Fm}_{\mathcal{L}^{\prime}})$ and $\boldrho:\alg{Fm}_{\mathcal{L}^{\prime}}\to\mathcal{P}(\alg{Fm}_{\mathcal{L}})$ such that, for every $\Gamma\cup{\varphi}\subseteq\alg{Fm}_{\mathcal{L}}$, $\psi\in\alg{Fm}_{\mathcal{L}^{\prime}}$

 \begin{gather*}
 \Gamma\vdash_{\logic{L}}\varphi\iff\boldtau(\Gamma)\vdash_{\logic{L}^{\prime}}\boldtau(\varphi) \tag{Eq1} \label{Eq1}\\
   \psi\dashv\vdash_{\logic{L}^{\prime}}\boldtau\boldrho(\psi)\tag{Eq2} \label{Eq2}.
\end{gather*}
\end{definition}

As we already sketched, subclassical logics come into pairs of logics of the form $(\logic{L}_{1},\logic{L}_{\ant})$ sharing the same defining three-valued algebra. In many cases, $\logic{L}_{1}$ lack any theorem, while $\logic{L}_{\ant}$ lacks antitheorems, and this suffices to show that a faithful mutual interpretation satisfying \ref{Eq1} and \ref{Eq2} between their consequence relations is unavailable. An example is provided by the pairs $(\logic{K}_{3},\logic{LP})$ and $(\logic{B},\logic{PWK})$. The situation is reversed when moving to the external versions. The next theorem indeed shows  that the external versions of each pair of the form $(\logic{L}_{1},\logic{L}_{\ant})$ are deductively equivalent.
\begin{theorem}\label{thm: deductive equivalence}
Let $\logic{L}_{1},\logic{L}_{\ant}$ be three-valued subclassical logics. Then $\logic{L}_{1}^{e},\logic{L}^{e}_{\ant}$ are pairwise deductively equivalent.
\end{theorem}
\begin{proof}
Recalling Definition \ref{def: deductive equivalence}  we need to prove that 
\begin{align*}
& \Gamma\vdash_{\logic{L}^{e}_{1}}\varphi\iff\Delta_{1}\Gamma\vdash_{\logic{L}^{e}_{\ant}}\Delta_{1}\varphi \tag{De1} \label{De1}\\
& \varphi\vdash_{\logic{L}^{e}_{\ant}}\Delta_{1}(\Delta_{1}\varphi\lor\Delta_{\ant}\varphi) \text{ and } \Delta_{1}(\Delta_{1}\varphi\lor\Delta_{\ant}\varphi)\vdash_{\logic{L}^{e}_{\ant}}\varphi.\tag{De2}\label{De2}
\end{align*}
We begin by proving (\ref{De1}). Assume $\Gamma\vdash_{\logic{L}^{e}_{1}}\varphi$ and let $h:\alg{Fm}^{e}\to\alg{A}^{e}$ be such that $h(\Delta_{1}\gamma)\in\{1,\ant\}$ for each $\gamma\in\Gamma$. By definition of $\Delta_{1}$, this entails $h(\gamma)=1$. So, by assumption, $h(\varphi)=1$. Thus, we conclude $h(\Delta\gamma)=1$, which proves $\Delta_{1}\Gamma\vdash_{\logic{L}^{e}_{\ant}}\Delta_{1}\varphi$. For the converse we reason by contraposition, so suppose there exists $h:\alg{Fm}^{e}\to \alg{A}^{e}$ such that $h(\gamma)=1$ for each $\gamma\in\Gamma$ and $h(\varphi)\neq 1$. . Thus, $h(\Delta_{1}\gamma)=1$ for each $\gamma\in\Gamma$ and $h(\Delta_{1}\varphi)=0$, i.e. $\Delta_{1}\Gamma\nvdash_{\logic{L}^{e}_{\ant}}\Delta_{1}\varphi$, as desired.

For (\ref{De2}), it is immediate to check that, for any $h:\alg{Fm}^{e}\to\alg{A}^{e}$, it holds $h(\varphi)\in\{1,\ant\}$ if and only if $h(\Delta_{1}(\Delta_{1}\varphi\lor\Delta_{\ant}\varphi))=1$.
 \end{proof}

Thus, adding external operators to a subclassical logic ensures the recovery of several desirable properties of classical logic, such as the Deduction-Detachment Theorem--and consequently, modus ponens and theoremhood--as well as algebraizability. Moreover, this method of constructing expansions naturally yields examples of pairwise deductively equivalent logics.

\end{document}